\documentclass[12pt]{amsart}
\usepackage{amsmath, amssymb, latexsym, amsthm, mathrsfs, geometry}

\newcommand{\nc}[1]{\newcommand{#1}}
\newtheorem{thm}{Theorem}[section]\nc{\bthm}{\begin{thm}} \nc{\ethm}{\end{thm}}
\newtheorem{prop}[thm]{Proposition}\nc{\bprp}{\begin{prop}} \nc{\eprp}{\end{prop}}
\newtheorem{fact}[thm]{Fact}\nc{\bfct}{\begin{fact}} \nc{\efct}{\end{fact}}
\newtheorem{prob}[thm]{Problem}\nc{\bprb}{\begin{prob}} \nc{\eprb}{\end{prob}}
\newtheorem{lem}[thm]{Lemma}\nc{\blem}{\begin{lem}} \nc{\elem}{\end{lem}}
\newtheorem{claim}[thm]{Claim}\nc{\bclm}{\begin{claim}} \nc{\eclm}{\end{claim}}
\newtheorem{cor}[thm]{Corollary}\nc{\bcor}{\begin{cor}} \nc{\ecor}{\end{cor}}
\newtheorem{conj}[thm]{Conjecture}\nc{\bcnj}{\begin{conj}} \nc{\ecnj}{\end{conj}}
\theoremstyle{definition}\newtheorem{defn}[thm]{Definition}\nc{\bdfn}{\begin{defn}} \nc{\edfn}{\end{defn}}
\theoremstyle{remark}
\newtheorem{rem}[thm]{Remark}\nc{\brem}{\begin{rem}}\nc{\erem}{\end{rem}}
\newtheorem{cnv}[thm]{Convention}\nc{\bcnv}{\begin{cnv}} \nc{\ecnv}{\end{cnv}}
\newtheorem{exam}[thm]{Example}\nc{\bexm}{\begin{exam}}\nc{\eexm}{\end{exam}}
\nc{\bpf}{\begin{proof}}\nc{\epf}{\end{proof}}\nc{\be}{\begin{enumerate}}\nc{\ee}{\end{enumerate}}
\nc{\bi}{\begin{itemize}}\nc{\itm}{\item}\nc{\ei}{\end{itemize}}

\nc{\ed}{
\subsection*{Credits}
All results in this paper were initially proved by its authors, who were not aware of the earlier works cited in the present version of this paper.
Initially, Theorem \ref{thm:Vec} was proved, chronologically earlier than the
other results, by the second and third named authors. The other results
were proved by the first and third named authors.
Our Theorem \ref{thm:Tor} may be new.

\end{document}
}

\begin{document}
\newcommand\chigen{\chi_{\mathrm{gen}}}
\newcommand\bbF{\mathbb{F}}
\newcommand\Un{\bigcup}
\newcommand\sm{\setminus}
\newcommand\sub{\subseteq}
\newcommand\bbZ{\mathbb{Z}}
\newcommand\N{\mathbb{N}}
\newcommand\bbH{\mathbb{H}}
\newcommand\Zp{\bbZ_{p}}
\newcommand\x{\times}
\newcommand\Zpp{\Zp\x\Zp}
\newcommand\PSL{\operatorname{PSL}}

\title{Monochromatic generating sets in groups and other algebraic structures}

\author{Noam Lifshitz, Itay Ravia, and Boaz Tsaban}

\address{Department of Mathematics, Bar-Ilan University, Ramat Gan 52900,
Israel}

\email{itayravia@gmail.com, lifshitz@gmail.com, tsaban@math.biu.ac.il}

\urladdr{http://math.biu.ac.il/\textasciitilde{}tsaban}

\begin{abstract}
The \emph{generating chromatic number} of
a group $G$, $\chigen(G)$, is the maximum number of colors $k$
such that there is a monochromatic generating set for each coloring
of the elements of $G$ in $k$ colors. If no such maximal $k$ exists,
we set $\chigen(G)=\infty$. 
Equivalently, $\chigen(G)$ is the maximal number $k$ such that there is no cover of $G$ by proper subgroups ($\infty$ if there is no such maximal $k$).

We provide characterizations, for arbitrary gruops,
in the cases $\chigen(G)=\infty$ and $\chigen(G)=2$. For nilpotent
groups (in particular, for abelian ones), all possible
chromatic numbers are characterized. Examples show that the characterization for nilpotent
groups do not generalize to arbitrary solvable groups. We conclude
with applications to vector spaces and fields.

\emph{Remark.} After completing this paper, we have learned that earlier works, using different terminology,
established most---if not all---of our results. 
We chose not to remove our proofs, but to provide references to all earlier proofs. 
Our Theorem \ref{thm:Tor} may be new, but it is not of great distance from known results, either.
\end{abstract}

\maketitle

\section{General properties and the case of infinite chromatic number}

Results in this section with no reference are folklore. 
All group theoretic background used in this paper can be found, for example, in \cite{DF}. We study the following Ramsey theoretic notion.

\bdfn
The \emph{generating chromatic number} of a group $G$, $\chigen(G)$,
is the maximum number of colors $k$ such that there is a monochromatic
generating set for each coloring of the elements of $G$ in $k$ colors.
If no such maximal $k$ exists, we set $\chigen(G)=\infty$.
\edfn
The following lemma gives a convenient reformulation of our definition.
\blem
\label{lem:sub-coloring-equiv} Let $G$ be a group. The following
are equivalent:
\begin{enumerate}
\item $\chigen(G)\ge k$.
\item For each cover $G=H_{1}\cup\cdots\cup H_{k}$ of $G$ by subgroups,
there is $i$ with $H_{i}=G$. 
\end{enumerate}
\elem
\begin{proof}
The first implication is obtained by giving elements of $H_{i}\sm(H_{1}\cup\cdots\cup H_{i-1})$
the color $i$, for each $i=1,\dots,k$.

For the second implication, replace each maximal monochromatic set
by the subgroup it generates.\end{proof}
\blem
\label{lem:ge2} Let $G$ be a group. 
\begin{enumerate}
\item $\chigen(G)\ge2$.
\item If $G$ is cyclic, then $\chigen(G)=\infty$.
\item If $G$ is finite noncyclic, then $\chigen(G)\le\left|G\right|-2$.
\end{enumerate}
\elem
\begin{proof}
(1) This is well known: If $G=H_{1}\cup H_{2}$ with $H_{1},H_{2}$
proper subgroups, then for $x\in G\sm H_{1}\sub H_{2}$ and $y\in G\sm H_{2}\sub H_{1}$,
we have that $xy\notin H_{1}\cup H_{2}$, a contradiction.

(2) A generator forms a monochromatic generating subset for each coloring.

(3) Give each element of $G\sm\left\{ e\right\} $ a unique color.
\end{proof}
Cyclic groups are not the only abelian groups with infinite generating
chromatic number. 
\bexm
$\chigen(\mathbb{Q},+)=\infty$.\eexm
\begin{proof}
This follows from the forthcoming Theorem \ref{thm:infchrome}; however,
we give a direct proof.

Assume that $\mathbb{Q}=H_{1}\cup\cdots\cup H_{k}$ with each $H_{i}$
a subgroup of $(\mathbb{Q},+)$. Then there is $i$ such that $H_{i}$
contains infinitely many rationals of the form $\frac{1}{n!}$. Then
$H_{i}=\mathbb{Q}$. Indeed, for each rational $\frac{l}{m}$, pick
$n\ge m$ with $\frac{1}{n!}\in H_{i}$. Then $\frac{n!}{m}$ is integer,
and therefore
\[
\frac{l}{m}=\frac{n!l}{m}\cdot\frac{1}{n!}\in H_{i}.
\]
\end{proof}

\blem[Cohn \cite{Cohn94}]\label{lem:quotient} 
If $H$ is a quotient (equivalently, a homomorphic image) of a group $G$, then $\chigen(G)\le\chigen(H)$.
\elem
\begin{proof}
Let $\varphi\colon G\to H$ be a surjective homomorphism. Give each
$g\in G$ the color of $\varphi(g)$. Then a monochromatic generating
subset of $G$ would map to a monochromatic generating subset of $H$.
\end{proof}

\bdfn
A cover $G=H_{1}\cup\cdots\cup H_{k}$ of a group $G$ by subgroups
is \emph{irredundant} if, for each $J\subsetneq\left\{ 1,\dots,k\right\} $,
$G\neq\Un_{i\in J}H_{i}$. \edfn
\blem
\label{lem:Irredundant} Let $G$ be a group and $k$ be a natural
number. Then:
\begin{enumerate}
\item If $\chigen(G)=k$, then there is an irredundant cover $G=H_{1}\cup\cdots\cup H_{k+1}$
of $G$ by subgroups.
\item If there is an irredundant cover $G=H_{1}\cup\cdots\cup H_{k+1}$
of $G$ by subgroups, then $\chigen(G)\le k$. \qed
\end{enumerate}
\elem

The following theorem is proved in \cite{HaberRosenfeld59} for finite groups, but the proof
actually treats the case of torsion groups.

\begin{prop}[Haber--Rosenfeld \cite{HaberRosenfeld59}]\label{prop:pLowerBound} 
Let $G$ be a finite (or just torsion) group,
and $p$ be the minimal order of an element of $G$. Then $\chigen(G)\ge p$.\end{prop}
\begin{proof}
Assume that $k:=\chigen(G)<\infty$. By Lemma \ref{lem:ge2}, $k\ge2$. 

Let $G=H_{1}\cup\cdots\cup H_{k+1}$ be an irredundant cover by subgroups.
Pick 
\[
g_{1}\in H_{1}\sm\left(H_{2}\cup\cdots\cup H_{k+1}\right),g_{2}\in H_{2}\sm\left(H_{1}\cup H_{3}\cup\cdots\cup H_{k+1}\right).
\]
For each $0\le n<p$, as $g_{1}^{n}\in H_{1}$ and $g_{2}\notin H_{1}$,
we have that$g_{1}^{n}g_{2}\notin H_{1}$. 

Assume that $p>k$. Then, by the pigeonhole principle, there are $0\le m<n<p$
and $2\le i\le k$ such that $g_{1}^{m}g_{2},g_{1}^{n}g_{2}\in H_{i}$.
Thus, $g_{1}^{n-m}\in H_{i}$. As the minimal order of an element
of $G$ is $p$ and $1\le n-m<p$, $g_{1}$ is a power of $g_{1}^{n-m}$,
and thus $g_{1}\in H_{i}$, a contradiction.\end{proof}
\begin{prop}[{Cohn \cite{Cohn94}}]\label{prop:CxG} 
Let $A,B$ be finite (or just torsion) groups such
that the order of each element of $A$ is coprime to the order of
each element of $B$. Then $\chigen(A\x B)=\min\left\{ \chigen(A),\chigen(B)\right\} $.\end{prop}
\begin{proof}
$(\le)$ By Lemma \ref{lem:quotient}.

$(\ge)$ Let $k\le\min\left\{ \chigen(A),\chigen(B)\right\} $, and
let $c\colon A\x B\to\left\{ 1,\dots,k\right\} $ be a coloring of
the elements of $A\x B$ in the colors $\left\{ 1,\dots,k\right\} $.
For each $a\in A$, define $c_{a}\colon B\to\left\{ 1,\dots,k\right\} $
by 
\[
c_{a}(b)=c(a,b).
\]
As $k\le\chigen(B)$, there is a monochromatic generating set $M_{a}$
of $B$, of some color $i_{a}$. Now, define a coloring $f\colon A\to\left\{ 1,\dots,k\right\} $
by 
\[
f(a)=i_{a}.
\]
Let $M\sub A$ be a monochromatic generating set, say of color $i$.
Let 
\[
I=\left\{ (a,b)\in A\x B\,:\, c(a,b)=i\right\} .
\]
We claim that $I$ generates $A\x B$. Indeed, fix $a\in M$ and $b\in M_{a}$.
Then $(a,b)\in I$. Let $m$ be the order of $b$. As $m$ is coprime
to the order of $a$, $(a,e)$ is a power of$(a^{m},e)=(a,b)^{m}$,
and thus belongs to $\left\langle I\right\rangle $. Similarly, $(e,b)\in\left\langle I\right\rangle $.
It follows that $M\x\left\{ e\right\} ,\left\{ e\right\} \x M_{a}\sub\left\langle I\right\rangle $,
and therefore $A\x\left\{ e\right\} ,\left\{ e\right\} \x B\sub\left\langle I\right\rangle $.
\end{proof}
The following result, proved in \cite{Neumann}(4.4), will make it
possible for us to reduce the case of arbitrary groups into the case
of finite ones.
\bthm[Neumann]\label{thm:Neumann}
For each irredundant cover $G=H_{1}\cup\cdots\cup H_{k}$ of a group
$G$ by subgroups, each of the subgroups $H_{1},\dots,H_{k}$ has
finite index in $G$.\ethm
\begin{cor}\label{cor:redFin} 
Let $G$ be a group with $\chigen(G)=k<\infty$.
Let $G=H_{1}\cup\cdots\cup H_{k+1}$ be an irredundant cover of $G$
by subgroups. There is a normal subgroup $N$ of $G$ such that:
\begin{enumerate}
\item $N$ has finite index in $G$;
\item $N\sub H_{1}\cap\cdots\cap H_{k}$;
\item $G/N=H_{1}/N\cup\cdots\cup H_{k}/N$ is an irredundant cover of the
finite group $G/N$; and
\item $\chigen(G/N)=k$.
\end{enumerate}
\end{cor}
\begin{proof}
(1,2) By Neumann's Theorem \ref{thm:Neumann}, the subgroup $H:=H_{1}\cap\cdots\cap H_{k+1}$,
being an intersection of finite index subgroups of $G$, has finite
index in $G$. By Poincar\'e's Theorem, there is a subgroup $N$ of
$H$ such that $N$ is a finite index, normal subgroup of $G$. 

(3) For $J\subsetneq\left\{ 1,\dots,k+1\right\} $, Let $g\in G\sm\Un_{i\in J}H_{i}$.
If $gN\in\Un_{i\in J}H_{i}/N$, say $gN\in H_{i}/N$, then $g\in h_{i}N\sub H_{i}$,
a contradiction.

(4) By Lemma \ref{lem:quotient}, Lemma \ref{lem:Irredundant} and
(3), $k=\chigen(G)\le\chigen(G/N)\le k$.
\end{proof}
We already obtain a characterization of the case where $\chigen(G)=\infty$. 
\bthm
\label{thm:infchrome} Let $G$ be a group. The following assertions
are equivalent:
\begin{enumerate}
\item $\chigen(G)=\infty$.
\item Every finite quotient of $G$ is cyclic.
\end{enumerate}
\ethm
\begin{proof}
$(1\Rightarrow2)$ Let $G/N$ be a finite quotient of $G$. By Lemma
\ref{lem:quotient}, $\chigen(G/N)=\infty$. By Lemma \ref{lem:ge2},
$G$ is cyclic.

$(2\Rightarrow1)$ Assume that $\chigen(G)<\infty$. Then, by Corollary
\ref{cor:redFin}, some finite quotient of $G$ has the same finite
generating chromatic number, and is thus not cyclic.
\end{proof}
It follows, for example, that infinite simple groups (e.g., $A_{\infty}$
or $\PSL_{n}(\bbF)$, $n\ge3$) have infinite generating chromatic
number.

\section{Nilpotent groups of finite chromatic number}

\blem[Cohn \cite{Cohn94}]
\label{lem:Zpp} $\chigen(\Zpp)=p$.\elem
\begin{proof}
$(\ge)$ Proposition \ref{prop:pLowerBound}.

$(\le)$ $\Zpp$ is the union of $p+1$ projective lines, the cyclic
subgroups generated by the $p+1$ elements $(0,1)$ and $(1,0),(1,1),\dots,(1,p-1)$. 
\end{proof}

Say that a group $G$ is \emph{projectively nilpotent} if every finite
quotient of $G$ is nilpotent. In particular, every nilpotent, quasi-nilpotent,
or pro-nilpotent group is projectively nilpotent. 
In the finite case, the following theorem was proved by Cohn \cite{Cohn94}.
The infinite case follows from Corollary \ref{cor:redFin}, which in turns follows
from Neumann's Theorem \ref{thm:Neumann}.

\bthm[Cohn, Neumann]
\label{thm:Nil} Let $G$ be a projectively nilpotent group with $\chigen(G)<\infty$.
Then $\chigen(G)$ is the minimum prime number $p$ such that $\Zpp$
is a quotient of $G$ (and there is such $p$).\ethm
\begin{proof}
Let $G$ be a projectively nilpotent group with $\chigen(G)<\infty$.
By Corollary \ref{cor:redFin}, we may assume that $G$ is a finite
nilpotent group. Write 
\[
G=P_{1}\x\cdots\x P_{n},
\]
where each $P_{i}$ is the $p_{i}$-Sylow subgroup of $G$. By Lemma
\ref{prop:CxG}, 
\[
\chigen(G)=\min\left\{ \chigen(P_{1}),\dots,\chigen(P_{n})\right\} .
\]
Let $P=P_{i}$ be with $\chigen(G)=\chigen(P_{i})$, and $p=p_{i}$. 

By Lemma \ref{prop:pLowerBound}, $\chigen(P)\ge p$. As $P$ is a
noncyclic, its rank $r$ is greater than $1$. As $P$ is a $p$-group,
$\Zp^{r}$ is a quotient of $P$, and in particular of $G$. 

Thus, $\chigen(G)=\chigen(P)=p$ and $\Zpp$ is a quotient of $G$.
By Lemmata \ref{lem:quotient} and \ref{lem:Zpp}, there is no smaller
prime $q$ such that $\bbZ_{q}\x\bbZ_{q}$ is a quotient of $G$.
\end{proof}

\section{General groups}

The forthcoming Theorem \label{thm:Tor} may be new.
Responding to a question of us,
Martino Garonzi came up with two alternative proofs of Theorem \label{thm:Tor}. One of these
proofs uses a powerful result of Tomkinson \cite[Lemma 3.2]{Tomkinson97}.\footnote{Garonzi 
also points a misprint there: It is necessary to assume there that the union of the 
$U_i$'s is different from $G$.}
We are greatful to Martino Garonzi for this information.

We will show, in Section \ref{sec:examples}, that Theorem \ref{thm:Nil}
does not generalize to arbitrary (or even finite solvable) groups.
In the present section, we provide a partial generalization of this
theorem to arbitrary torsion groups. Recall, from Lemma \ref{prop:pLowerBound},
that the premise in the following theorem implies that $\chigen(G)\ge p$.
\bthm\label{thm:Tor} 
Let $G$ be a nontrivial torsion group, and let $p$
be the minimal order of a nonidentity element of $G$. The following
assertions are equivalent:
\begin{enumerate}
\item $\chigen(G)=p$.
\item $\Zpp$ is a quotient of $G$.
\end{enumerate}
\ethm
\begin{proof}
$(2)\Rightarrow(1)$: Proposition \ref{prop:pLowerBound}, Lemma \ref{lem:quotient}
and Lemma \ref{lem:Zpp}.

$(1)\Rightarrow(2)$: Assume that $G/N$ is a quotient of $G$. The
order of each element $gN\in G/N$ is $\ge p$. Indeed, the order
of $gN$ divides the order of $g$, and is therefore the order of
a power of $g$, which in turn is $\ge p$. Thus, let 
\[
G=H_{1}\cup\cdots\cup H_{p+1}
\]
be an irredundant cover by proper subgroups. Then, by Corollary \ref{cor:redFin},
we know $G$ has a finite quotient $G/N$ with $\chigen\left(G/N\right)=\chigen\left(G\right)$. 
Since each element of $G/N$ is of order $\geq p$ we may assume
$G$ is finite. We now have by Cauchy's theorem that $\left|G\right|$
and $i$ are relatively prime for each $i<p$. 

Let $H$ be a subgroup of $G$. Then $\left[G:H\right]$ divides $\left|G\right|$,
and by Cauchy's theorem, there is $g\in G$ whose order divides $\left[G:H\right]$.
In particular, $\left[G:H\right]\ge p$ and $\left|H\right|\le\left|G\right|/p$.
Assume that, for each $i=1,\dots,p+1$, $\left[G:H_{i}\right]\ge p+1$.
Then, as $e\in H_{1}\cap\cdots\cap H_{p+1}$,
\[
\left|G\right|=\left|H_{1}\cup\cdots\cup H_{p+1}\right|<\left|H_{1}\right|+\cdots+\left|H_{p+1}\right|\le(p+1)\cdot\frac{\left|G\right|}{p+1}=\left|G\right|,
\]
a contradiction. Thus, we may assume that $\left[G:H_{1}\right]=p$.
As $p$ is the minimal prime divisor of $\left|G\right|$, $H_{1}$
is a \emph{normal} subgroup of $G$.

Fix $i\in\left\{ 2,\dots,p+1\right\} $. Then $p=\left[G:H_{1}\right]=\left[G:H_{i}H_{1}\right]\cdot\left[H_{i}H_{1}:H_{1}\right]$.
As $p$ is prime and $H_{i}\neq H_{1}$ we have by the Second Isomorphism
Theorem that $p=\left[H_{i}H_{1}:H_{1}\right]=\left[H_{i}:H_{i}\cap H_{1}\right]$.
It follows that 
\[
\left|H_{i}\sm H_{1}\right|=\left|H_{i}\sm(H_{i}\cap H_{1})\right|=\frac{p-1}{p}\cdot\left|H_{i}\right|.
\]
Thus,
\begin{eqnarray*}
\frac{p-1}{p}\cdot\left|G\right| & = & \left|G\sm H_{1}\right|=\left|(H_{2}\sm H_{1})\cup(H_{3}\sm H_{1})\cup\cdots\cup(H_{p+1}\sm H_{1})\right|\leq\\
 & \leq & \left|H_{2}\sm H_{1}\right|+\left|H_{3}\sm H_{1}\right|+\cdots+\left|H_{p+1}\sm H_{1}\right|=\\
 & = & \frac{p-1}{p}(\left|H_{2}\right|+\left|H_{3}\right|+\cdots+\left|H_{p+1}\right|).
\end{eqnarray*}
As $\left|H_{i}\right|\le\left|G\right|/p$ for each $i$,
\[
\left|G\right|\leq\left|H_{2}\right|+\cdots+\left|H_{p+1}\right|\leq\frac{\left|G\right|}{p}+\cdots+\frac{\left|G\right|}{p}=\left|G\right|.
\]
Thus $\left|H_{2}\right|=\left|H_{3}\right|=\cdots=\left|H_{p+1}\right|=\left|G\right|/p$, 
and therefore $H_{2},\dots,H_{p+1}$ are also normal subgroups of $G$. 

Let $N=H_{1}\cap H_{2}\cap\cdots\cap H_{p+1}$. By moving to $G/N=H_{1}/N\cup\cdots\cup H_{p+1}/N$
if needed, we may assume that $H_{1}\cap\cdots\cap H_{p+1}=\left\{ e\right\} $. 

We claim that all elements of $G$ are of order $p$.

Let $g\in G$ be with $g^{p}\neq e$. We may assume that $g\in H_{1}$.
Then $g^{p}\in H_{1}$. As $H_{1}\cap\cdots\cap H_{p+1}=\left\{ e\right\} $,
we may assume that $g^{p}\notin H_{2}$. Let $h\in H_{2}\sm(H_{1}\cup H_{3}\cup\cdots\cup H_{p+1})$.
At least two of the $p+2$ elements 
\[
g,h,hg,hg^{2},\dots,hg^{p}
\]
are in the same $H_{i}$. 

If $i=1$, then $g,hg^{j}\in H_{1}$ for some $j$ and therefore $h\in H_{1}$,
a contradiction. 

If $i=2$, then similarly $g^{j}\in H_{2}$ for some $j\in\left\{ 0,\dots,p\right\} $.
As $g^{p}\notin H_{2}$, $j<p$. By Cauchy's theorem, $p$ is the
minimal prime factor of $\left|G\right|$. Thus, $j$ is coprime to
$\left|G\right|$, and in particular to the order of $g$. Thus, $g\in\left\langle g^{j}\right\rangle \sub H_{2}$,
and therefore $g^{p}\in H_{2}$, a contradiction. 

Thus, $i>2$. We may assume that $i=3$. If $g,hg^{j}\in H_{3}$ for
some $j$, then $h\in H_{3}$, a contradiction. Thus, $hg^{j},hg^{k}\in H_{3}$
for some $0\le j<k\le p$. Then $g^{k-j}\in H_{3}$. As $k-j<p$,
$g\in H_{3}$ and thus also $h\in H_{3}$, a contradiction.

Thus, all elements of $G$ are of order $p$, and in particular $G$
is a $p$-group. As $\chigen(G)=p<\infty$, $G$ is not cyclic, and
therefore $\Zpp$ is a quotient of $G$.
\end{proof}

In the case of $2$ colors, we obtain a completely general characterization.

\bthm[Scorza \cite{Scorza26}]
\label{thm:twocolors} Let $G$ be a group. The following assertions
are equivalent:
\begin{enumerate}
\item $\chigen(G)=2$.
\item $\bbZ_{2}\x\bbZ_{2}$ is a quotient of $G$.
\end{enumerate}
\ethm
\begin{proof}
$(2\Rightarrow1)$ Lemmata \ref{lem:quotient}, \ref{lem:Zpp}, and
\ref{lem:ge2}.

$(1\Rightarrow2)$ By Corolllary\ref{cor:redFin}, we may assume that
$G$ is a nontrivial finite group. By Proposition \ref{prop:pLowerBound},
the minimal order of a nonidentity element of $G$ is $2$. Apply
Theorem \ref{thm:Tor}.
\end{proof}

We will show, in Example \ref{exm:Z3Z} below, that a $3$-colors
version of Theorem \ref{thm:twocolors} is not true in general.

\section{examples\label{sec:examples}}

All results in this section can, alternatively, be derived from Tomkinson's treatment of the general finite solvable case \cite{Tomkinson97}.

One may wonder whether Theorem \ref{thm:Nil} holds for arbitrary
groups $G$, or at least for arbitrary solvable groups. We will show
that the answer is negative, in a strong sense.

Recall that the \emph{Frattini subgroup} of a group $G$, $\phi\left(G\right)$,
is defined to be the intersection of all maximal subgroups of $G$,
or $G$ if $G$ has no maximal subgroups. Recall that the elements
of $\phi(G)$ are all the non-generators of $G$, that is, all elements
$g\in G$ such that $G=\left\langle g,H\right\rangle $ implies $H=G$. 
\blem
\label{lem:frattini} Let $H\trianglelefteq G$ be a subgroup of the
Frattini subgroup $\phi\left(G\right)$. Then $\chigen\left(G\right)=\chigen\left(G/H\right)$.\elem
\begin{proof}
$\left(\leq\right)$ By Lemma \ref{lem:quotient}.

$\left(\ge\right)$ Let $k=\chigen\left(G\right)$, and let $G=H_{1}\cup H_{2}\cup\cdots\cup H_{k+1}$
be an irredundant cover of $G$ by subgroups. As each $H_{i}$ is
of finite index, we can replace each $H_{i}$ by a group of minimal
index containing $H_{i}$. Thus, for each $i$, we may asume that
$H_{i}$ is maximal, and therefore $H\leq H_{i}$. As 
\[
G/H=H_{1}/H\cup H_{2}/H\cup\cdots\cup H_{k+1}/H,
\]
we have that $\chigen\left(G/H\right)\le k$.\end{proof}
\blem
$\chigen\left(\bbZ_{p}\rtimes_{\varphi}\mathbb{Z}_{n}\right)=p$ if
$\varphi$ is nontrivial. In general, $\chigen\left(\bbZ_{p}\rtimes_{\varphi}\mathbb{Z}_{n}\right)$
is either $p$ or $\infty$.
\elem
\begin{proof}
Write $\bbZ_{p}\rtimes_{\varphi}\mathbb{Z}_{n}=\left\langle x\right\rangle \rtimes\left\langle y\right\rangle $.
$\left\langle x\right\rangle $ acts on the subgroups of $\bbZ_{p}\rtimes_{\varphi}\mathbb{Z}_{n}$
by conjugation. We claim that $\left|\left[\left\langle y\right\rangle \right]\right|=p$.
Since $\left\langle y\right\rangle $ is maximal non-normal, $N_{\bbZ_{p}\rtimes_{\varphi}\mathbb{Z}_{n}}\left(\left\langle y\right\rangle \right)=\left\langle y\right\rangle $
and therefore $N_{\left\langle x\right\rangle }\left(\left\langle y\right\rangle \right)=\left\langle y\right\rangle \cap\left\langle x\right\rangle =\left\{ 1\right\} $.
Thus 
\[
\left|\left[\left\langle y\right\rangle \right]\right|=\left[\left\langle x\right\rangle :N_{\left\langle x\right\rangle }\left(\left\langle y\right\rangle \right)\right]=\left[\left\langle x\right\rangle :\left\{ 1\right\} \right]=p.
\]
Since $\left(\left\langle y\right\rangle \cap\left\langle y^{x}\right\rangle \right)^{x}$
is a subgroup of the cyclic group $\left\langle y^{x}\right\rangle $,
with the same order as $\left\langle y\right\rangle \cap\left\langle y^{x}\right\rangle $,
we have that $\left\langle y\right\rangle \cap\left\langle y^{x}\right\rangle =\left(\left\langle y\right\rangle \cap\left\langle y^{x}\right\rangle \right)^{x}$.
Thus, $x,y\in N_{\bbZ_{p}\rtimes_{\varphi}\mathbb{Z}_{n}}\left(\left\langle y\right\rangle \cap\left\langle y^{x}\right\rangle \right)$,
and therefore $\left\langle y\right\rangle \cap\left\langle y^{x}\right\rangle \trianglelefteq\bbZ_{p}\rtimes_{\varphi}\mathbb{Z}_{n}$.
Let $m=\left|\left\langle y\right\rangle \cap\left\langle y^{x}\right\rangle \right|$.
Then $\left\langle y\right\rangle \cap\left\langle y^{x}\right\rangle =\left\langle y^{\frac{n}{m}}\right\rangle $.
Since $\left[x,y^{\frac{n}{m}}\right]\in\left\langle y^{\frac{n}{m}}\right\rangle \cap\left\langle x\right\rangle =1$,
$y^{\frac{n}{m}}$ commutes with both $x$ and $y$, and therefore
$y^{\frac{n}{m}}\in Z\left(\bbZ_{p}\rtimes_{\varphi}\mathbb{Z}_{n}\right)$.
Let $k=\chigen\left(\bbZ_{p}\rtimes_{\varphi}\mathbb{Z}_{n}\right)$,
and let 
\[
\bbZ_{p}\rtimes_{\varphi}\mathbb{Z}_{n}=H_{1}\cup H_{2}\cup\cdots\cup H_{k+1}
\]
be an irredundant cover of $\bbZ_{p}\rtimes_{\varphi}\mathbb{Z}_{n}$
by subgroups. Since each element $A\in\left[\left\langle y\right\rangle \right]$
is cyclic and maximal, 
\[
\left[\left\langle y\right\rangle \right]\subseteq\left\{ H_{1},H_{2},\dots,H_{k+1}\right\} .
\]
Assume that $\left[\left\langle y\right\rangle \right]=\left\{ H_{1},H_{2},\dots,H_{k+1}\right\} $.
Then $x\in H_{1}\cup H_{2}\cup\cdots\cup H_{k+1}=\bigcup_{g\in\left\langle x\right\rangle }\left\langle g^{-1}yg\right\rangle $,
and therefore $x=g^{-1}y^{m}g$ for some $g\in\left\langle x\right\rangle ,m\in\bbZ$,
and thus $gxg^{-1}\in\left\langle y\right\rangle \cap\left\langle x\right\rangle =\left\{ 1\right\} $,
a contradiction. Thus, $p\leq k$. 

To prove that $p\ge k$, let $H_{i}=\left\langle y^{\left(x^{i}\right)}\right\rangle $
for $i=1,\ldots,p$, and let $H_{p+1}=\left\langle x\right\rangle \left\langle y^{\frac{n}{m}}\right\rangle $. 

We claim that all intersections $H_{i}\cap H_{j}$, $i\neq j$, are
contained in $\left\langle y^{\frac{n}{m}}\right\rangle $. First,
let $i,j\leq p$. As in the case $i=1,j=2$, we have that $H_{i}\cap H_{j}\trianglelefteq\bbZ_{p}\rtimes_{\varphi}\mathbb{Z}_{n}$,
and therefore $H_{i}\cap H_{j}\leq H_{i}^{g}$ for each $g\in\left\langle x\right\rangle $.
In particular, $H_{i}\cap H_{j}\leq H_{1}\cap H_{2}=\left\langle y\right\rangle \cap\left\langle y^{x}\right\rangle =y^{\frac{n}{m}}$.
Next, let $i<j=p+1$ . Assume that $g\in H_{i}\cap H_{p+1}$. Then
\[
^{x^{i}}g\in^{x^{i}}H_{i}\cap H_{p+1}=H_{1}\cap H_{p+1}.
\]
Write $^{x^{i}}g=x^{a}y^{\frac{nb}{m}}$. Then $x^{a}y^{\frac{nb}{b}}\in\left\langle y\right\rangle $,
and therefore $x^{a}\in\left\langle y\right\rangle \cap\left\langle x\right\rangle =\left\{ 1\right\} $.
Thus, $^{x^{i}}g=y^{\frac{nb}{m}}$. Since $y^{\frac{nb}{m}}$ is
in the center of $\bbZ_{p^{n}}\rtimes_{\varphi}\bbZ_{m}$, we have
that $g\in\left\langle y^{\frac{n}{m}}\right\rangle $. 
Thus,
\begin{eqnarray*}
\lefteqn{\left|H_{1}\setminus\left\langle y^{\frac{n}{m}}\right\rangle \cup H_{2}\setminus\left\langle y^{\frac{n}{m}}\right\rangle \cup\cdots\cup H_{p+1}\right|=}\\
& = & 
\left|H_{1}\setminus\left\langle y^{\frac{n}{m}}\right\rangle \right|+\left|H_{2}\setminus\left\langle y^{\frac{n}{m}}\right\rangle \right|+\ldots+\left|H_{p+1}\right| \geq\\
&\geq & \left(n-\frac{n}{m}\right)p+\frac{pn}{m}=np.
\end{eqnarray*}
As
\[
H_{1}\setminus\left\langle y^{\frac{n}{m}}\right\rangle \cup H_{2}\setminus\left\langle y^{\frac{n}{m}}\right\rangle \cup\cdots\cup H_{p+1}\sub H_{1}\cup H_{2}\cup\cdots\cup H_{p+1},
\]
we have that $\left|H_{1}\cup H_{2}\cup\cdots\cup H_{p+1}\right|\ge np$,
and therefore $H_{1}\cup H_{2}\cup\cdots\cup H_{p+1}=\bbZ_{p}\rtimes_{\varphi}\mathbb{Z}_{n}$.

Finally, if $\varphi$ is trivial, then $\bbZ_{p}\rtimes_{\varphi}\mathbb{Z}_{n}$
is abelian and Theorem \ref{thm:Nil} applies.\end{proof}
\blem
\label{lem:p-power} Let $n$ be a natural number and $p$ be a prime
number. If $\varphi$ is non-trivial, then $\chigen\left(\bbZ_{p^{n}}\rtimes_{\varphi}\bbZ_{m}\right)=p$.\elem
\begin{proof}
Write $\bbZ_{p}\rtimes_{\varphi}\mathbb{Z}_{n}=\left\langle x\right\rangle \rtimes\left\langle y\right\rangle $.
We first show $\left\langle x^{p}\right\rangle \le\phi\left(\bbZ_{p^{n}}\rtimes_{\varphi}\bbZ_{m}\right)$.
Let $H$ be maximal in $\bbZ_{p^{n}}\rtimes_{\varphi}\bbZ_{m}$ then
$H\left\langle x^{p}\right\rangle $ is a subgroup containing $H$.
Since $\left\langle x^{p}\right\rangle \trianglelefteq\bbZ_{p^{n}}\rtimes_{\varphi}\bbZ_{m}$,
$H\le H\left\langle x^{p}\right\rangle \le\bbZ_{p^{n}}\rtimes_{\varphi}\bbZ_{m}$.
Assume that $\left\langle x^{p}\right\rangle \nleq H$. Then $H\left\langle x^{p}\right\rangle =\bbZ_{p^{n}}\rtimes_{\varphi}\bbZ_{m}$,
and therefore $x=hx^{ap}$ for some $h\in H,a\in\bbZ$. Thus, $x^{1-ap}\in H$. 
Since $1-ap$ and $p^{n}$ are relatively prime, we have that $x\in H$,
a contradiction. Thus, $\left\langle x^{p}\right\rangle \leq H$. 

By Lemma \ref{lem:frattini}, 
\[
\chigen\left(\bbZ_{p^{n}}\rtimes_{\varphi}\bbZ_{m}\right)=\chigen\left(\bbZ_{p^{n}}\rtimes_{\varphi}\bbZ_{m}/\left\langle x^{p}\right\rangle \right)=\chigen\left(\bbZ_{p}\rtimes_{\varphi'}\bbZ_{m}\right)\in\left\{ p,\infty\right\} .
\]
Since $\bbZ_{p^{n}}\rtimes_{\varphi}\bbZ_{m}$ is non-abelian, and
therefore non-cyclic, $\chigen(\bbZ_{p^{n}}\rtimes_{\varphi}\bbZ_{m})<\infty$,
and therefore $\chigen\left(\bbZ_{p^{n}}\rtimes_{\varphi}\bbZ_{m}\right)=p$.
\end{proof}

\begin{cor}
$\chigen\left(\bbZ_{p^{n}}\rtimes_{\varphi}\bbZ_{m}\right)=\infty$
if $m,p$ are relatively prime and $\varphi$ is trivial. Otherwise,
$\chigen\left(\bbZ_{p^{n}}\rtimes_{\varphi}\bbZ_{m}\right)=p$.
\end{cor}
\begin{proof}
Lemmata \ref{thm:Nil} and \ref{lem:p-power}.\end{proof}
\bthm
\label{thm:semiprod}Let $G=\bbZ_{m}\rtimes_{\varphi}\bbZ_{n}$, and
write $\bbZ_{m}=\bbZ_{p_{1}^{d_{1}}}\times\bbZ_{p_{2}^{d_{2}}}\times\ldots\times\bbZ_{p_{r}^{d_{r}}}$
with $p_{1}<p_{2}<\cdots<p_{r}$. If
\[
\left\{ p_{i}\mid\left(m,n\right)\,:\, i=1\ldots r\right\} \cup\left\{ p_{i}\mid m\,:\,\text{ \ensuremath{\bbZ_{n}}\mbox{ acts on }\ensuremath{\bbZ_{p_{i}^{d_{i}}}}\mbox{ nontrivially}}\right\} =\varnothing,
\]
then $\chigen\left(G\right)=\infty$. Otherwise,
\[
\chigen\left(G\right)=\min\left\{ p_{i}\mid\left(m,n\right)\,:\, i=1,\ldots,r\right\} \cup\left\{ p_{i}\mid m\,:\,\text{ \ensuremath{\bbZ_{n}}\mbox{ acts on }\ensuremath{\bbZ_{p_{i}^{d_{i}}}}\mbox{ nontrivially}}\right\} .
\]
\ethm
\begin{proof}
By induction on $m$. If $m$ is prime, then Lemma \ref{lem:p-power}
applies. 

If $\bbZ_{n}$ acts trivialy on $\bbZ_{p_{1}^{d_{1}}}$and $p_{1}\nmid n$,
then $\bbZ_{p_{1}^{d_{1}}}\leq Z\left(G\right)$. Thus, 
\[
\left(\bbZ_{p_{1}^{d_{1}}}\times\bbZ_{p_{2}^{d_{2}}}\times\cdots\times\bbZ_{p_{r}^{d_{r}}}\right)\rtimes\bbZ_{n}=\bbZ_{p_{1}^{d_{1}}}\times\left(\left(\bbZ_{p_{2}^{d_{2}}}\times\cdots\times\bbZ_{p_{r}^{d_{r}}}\right)\rtimes\bbZ_{n}\right).
\]
By proposition \ref{prop:CxG}, since $p_{_{1}}^{d_{1}}$ and $mn/p_{_{1}}^{d_{1}}$
are relatively prime, 
\begin{eqnarray*}
\lefteqn{\chigen\left(\bbZ_{p_{1}^{d_{1}}}\times\left(\left(\bbZ_{p_{2}^{d_{2}}}\times\cdots\times\bbZ_{p_{r}^{d_{r}}}\right)
\rtimes_{\varphi}\bbZ_{n}\right)\right)=}\\
& = & \min\left\{ \chigen\left(\left(\bbZ_{p_{2}^{d_{2}}}\times\cdots\times\bbZ_{p_{r}^{d_{r}}}\right)\rtimes_{\varphi}\bbZ_{n}\right),\chigen\left(\bbZ_{p_{1}^{d_{1}}}\right)\right\} .
\end{eqnarray*}
By the induction hypothesis, we are done in this case. 

If $\bbZ_{n}$ acts nontrivialy on $\bbZ_{p_{1}^{d_{1}}}$or $p_{1}\mid n$,
we have that 
\[
\left(\bbZ_{p_{1}^{d_{1}}}\times\bbZ_{p_{2}^{d_{2}}}\times\cdots\times\bbZ_{p_{r}^{d_{r}}}\right)\rtimes_{\varphi}\bbZ_{n}/\left(\bbZ_{p_{2}^{d_{2}}}\times\cdots\times\bbZ_{p_{r}^{d_{r}}}\right)\cong\bbZ_{p_{1}^{d_{1}}}\rtimes_{\varphi}\bbZ_{n}.
\]
Since $\bbZ_{p_{1}^{d_{1}}}\rtimes_{\varphi}\bbZ_{n}$ is non-cyclic,
\[
\chigen\left(\left(\bbZ_{p_{1}^{d_{1}}}\times\bbZ_{p_{2}^{d_{2}}}\times\cdots\times\bbZ_{p_{r}^{d_{r}}}\right)\rtimes\bbZ_{n}\right)\leq\chigen\left(\bbZ_{p_{1}^{d_{1}}}\rtimes_{\varphi}\bbZ_{n}\right)=p_{1}.
\]
It remains to show that $\chigen\left(\bbZ_{m}\rtimes\bbZ_{n}\right)\geq p_{1}$.
Write $\bbZ_{m}\rtimes\bbZ_{n}=\left\langle x\right\rangle \rtimes\left\langle y\right\rangle $
let 
\[
H_{1}\cup H_{2}\cup\cdots\cup H_{k}=G
\]
 be irrudandent . If $x\in H_{1}\cap H_{2}\cap\cdots\cap H_{k}$ then
\[
\left\langle y\right\rangle \cong G/\left\langle x\right\rangle =H_{1}/\left\langle x\right\rangle \cup\cdots\cup H_{k}/\left\langle x\right\rangle 
\]
Thus, for some $i$, $H_{i}/\left\langle x\right\rangle =G/\left\langle x\right\rangle $.
Since $\left\langle x\right\rangle \leq H_{i}$, we have that $H_{i}=G$
. Assume that $x\notin H_{1}$. Take $y\in H_{1}\setminus\left(H_{2}\cup\cdots\cup H_{k}\right)$
then at least two of the elements $y,x,xy,x^{2}y,x^{3}y,\ldots,x^{p-1}y$
are contained in the same $H_{i}$. If $x,x^{j}y\in H_{i}$ then $x,y\in H_{i}$
but since $y\in H_{i}$, we have that $i=1$. But then $x\notin H_{i}$,
a contradiction. If $x^{j}y,x^{t}y\in H_{i}$ with $j<t$ we have
$x^{t-j}\in H_{i}$. Since $t-j<p_{1}$, $t-j$ and $n$ are relatively
prime. Thus, $x\in H_{i}$ and we obtain a contradiction as before.\end{proof}
\begin{cor}
\label{cor:Dn}$\chigen\left(D_{2n}\right)$ is the minimal prime
factor of $n$.\end{cor}
\begin{proof}
$D_{2n}=\left\langle r,s\,:\, r^{n}=s^{2}=e,\, srs=r^{-1}\right\rangle \cong\bbZ_{n}\rtimes\bbZ_{2}$.
If $n$ is even, then $\chigen\left(\bbZ_{n}\rtimes\bbZ_{2}\right)=2$
by the preceding theorem. Otherwise, let $p$ be the minimal prime
factor of $n$, and let $\bbZ_{p^{i}}$ be the Sylow $p$-group of
$\bbZ_{n}$. Let $y\in\bbZ_{p^{i}}\setminus\left\{ e\right\} $ then
$sys^{-1}=y^{-1}\neq y$ . Thus, $\bbZ_{2}$ acts nontrivially on
$\bbZ_{p^{i}}$ and the result follows from the preceding theorem.\end{proof}
\bexm
\label{exm:Z3Z} $\chigen(\bbZ_{3}\rtimes\bbZ)=3$, but $\Zpp$ is
not a quotient of $\bbZ_{3}\rtimes\bbZ$ for any $p$.\eexm
\begin{proof}
As the automorphism group of $\bbZ_{3}$ has order $2$, the action
of $2\bbZ$ on $\bbZ_{3}$ by conjugation is trivial, that is, $2\bbZ$
is in the center of $\bbZ_{3}\rtimes\bbZ$. Now, $\bbZ_{3}\rtimes\bbZ/2\bbZ=\bbZ_{3}\rtimes\bbZ_{2}=D_{6}$.
By Lemma \ref{lem:quotient} and Corollary \ref{cor:Dn}, $\chigen(\bbZ_{3}\rtimes\bbZ)\le\chigen(D_{6})=3$.
By Theorem \ref{thm:twocolors}, it remains to prove that $\Zpp$
is not a quotient of $\bbZ_{3}\rtimes\bbZ$ for any $p$. 

Case 1: $p\neq3$. Assume that $\bbZ_{3}\rtimes\bbZ/H=\bbZ_{p}\x\bbZ_{p}$.
Write $\bbZ_{3}\rtimes\bbZ=\left\langle x\right\rangle \rtimes\left\langle y\right\rangle $.
As $x^{3}=e$, the order of $xH$ divides both $3$ and $p$, and
is therefore $1$. Thus, $x\in H$, that is, $\bbZ_{3}\le H$, and
by the Third Isomorphism Theorem, 
\[
\bbZ/(H/\bbZ_{3})=(\bbZ_{3}\rtimes\bbZ/\bbZ_{3})/(H/\bbZ_{3})=\bbZ_{3}\rtimes\bbZ/H=\Zpp
\]
is a quotient of a cyclic group, and thus cyclic, a contradiction. 

Case 2: $p=3$. Assume that $\bbZ_{3}\rtimes\bbZ/H=\bbZ_{3}\x\bbZ_{3}$.
Write $\bbZ_{3}\rtimes\bbZ=\left\langle x\right\rangle \rtimes\left\langle y\right\rangle $.
Then $y^{3}\in H$. Thus, $x^{2}y^{3}=xy^{3}x^{-1}\in H$, and therefore
$x^{2}\in H$. It follows that $\bbZ_{3}\le H$, and we get a contradiction
as in Case 1.
\end{proof}
We have proved that the generating chromatic number of a nilpotent
group is either $\infty$ or prime. One may wonder whether at least
this can be generalized, for solvable groups at least. This is not
the case.
\bexm
$\chigen(A_{4})=4$.
\eexm
\begin{proof}
$A_{4}$ has four maximal subgroups of order $3$ and one maximal
subgroup of order $4$. Together, these groups cover $A_{4}$. Thus,
$\chigen(A_{4})<5$. On the other hand, no four of these groups cover
$A_{4}$, since the identity elements belongs to them all, and thus
the cardinality of their union is at most $1+3(3-1)+(4-1)=10<12$.
\end{proof}
It can be shown that $\chigen(S_{4})=3$ by verifying that $3\le\chigen(S_{4})$
and using that $S_{3}=D_{6}$ is a quotient of $S_{4}$.

\section{Vector spaces and fields}
\bdfn
The \emph{generating chromatic number} of a vector space $V$, $\chigen(V)$,
is the maximum number of colors $k$ such that there is a monochromatic
spanning set for each coloring of the elements of $V$ in $k$ colors.
If no such maximal $k$ exists, we set $\chigen(G)=\infty$.
\edfn
The proofs of the following two lemmata are similar to those of Lemmata
\ref{lem:quotient} and \ref{lem:sub-coloring-equiv}.
\blem
\label{lem:quotient-1}If a vector space $U$ is a linear image of
a vector space $V$, then $\chigen(V)\le\chigen(U)$. \qed
\elem

\blem
\label{lem:sub-coloring-equiv-1} Let $V$ be a vector space . The
following are equivalent:
\begin{enumerate}
\item $\chigen\left(V\right)\geq k$.
\item For each cover $V=V_{1}\cup\cdots\cup V_{k}$ of $V$ by subspaces,
there is $i$ with $V_{i}=V$. \qed
\end{enumerate}
\elem
The proof similar to Lemma \ref{lem:Zpp} also establishes the following.
\blem
\label{lem:Zpp-1}Let $\bbF$ be a finite field. The vector space
$\bbF^{2}$ satisfies $\chigen(\bbF^{2})\leq\left|\bbF\right|$.
\elem
\begin{proof}
We find $\left|\bbF\right|+1$ proper subspaces that cover $\bbF^{2}$.
For each $\alpha\in\bbF$, let $V_{\alpha}=\text{ span }\left\{ \left(1,\alpha\right)\right\} $.
Let $V=\text{span }\left\{ \left(0,1\right)\right\} $. To see that
$\bbF^{2}=V\cup\bigcup_{\alpha\in\bbF}V_{\alpha}$, let $\left(\alpha,\beta\right)\in\bbF^{2}$.
If $\alpha=0$ then $\left(\alpha,\beta\right)\in V$, and if not,
then $\left(\alpha,\beta\right)\in\text{ span }\left\{ \left(1,\alpha^{-1}\beta\right)\right\} \in V_{\alpha^{-1}\beta}$.
\end{proof}

Clearly, for vector spaces $V$ of dimension $1$, $\chigen(V)=\infty$.

The inequality $(\ge)$ in the following theorem was proved in 
Bialynicki-Birula--Browkin--Schinzel \cite{BBS59}.
The other inequality was proved, e.g., in Khare \cite{Khare}.

\bthm[Bialynicki-Birula--Browkin--Schinzel, Khare]
\label{thm:Vec}Let $V$ be a vector space of dimension $\ge2$ over
a field $\bbF$. Then $\chigen\left(V\right)=\left|\bbF\right|$ if
$\bbF$ is finite, and $\infty$ otherwise.\ethm
\begin{proof}
$\left(\leq\right)$ By Lemmata \ref{lem:Zpp-1} and \ref{lem:quotient-1}.

$\left(\geq\right)$ Similar to the proof of Proposition \ref{prop:pLowerBound}.
\end{proof}

\bthm[Bialynicki-Birula--Browkin--Schinzel \cite{BBS59}]
Let $\bbF$ be a field. For each coloring of the elements of $\bbF$
in finitely many colors, there a monochromatic set generating $\bbF$
as a field.
\ethm

\begin{proof}
Assume that there is a finite coloring $c$ with no monochromatic
generating set, with a finite, minimal number of colors. For each
color $i$, let $\bbF_{i}$ be the subfield generated by the elements
of color $i$. Let $\bbH=\bigcap_{i}\bbF_{i}$.

Since, as groups, $\left(\bbF_{i},+\right)\leq\left(\bbF,+\right)$
for all colors $i$, we have by Neumann's Theorem \ref{thm:Neumann}
that $\left(\bbH,+\right)$ is of finite index in $\left(\bbF,+\right)$. 

If $\bbH$ is finite, then $\bbF$ is finite, and then the multiplicative
group $\bbF^{*}$ is cyclic. In particular, $\bbF$ is genrated, as
a field, by a single element, which in turn consitututes a monochromatic
generating set, a contradiction. 

If $\bbH$ is infinite, then as $\bbF$ and each $\bbF_{i}$ are vector
spaces over the infinite field $\bbH$, we have by Thorem \ref{thm:Vec}
that there is a monochromatic set generating $\bbF$ as a vector space
of $\bbH$, and, in particular, as a field, a contradiction.
\end{proof}

\section{Open problems}

Tomkinson \cite{Tomkinson97} provides a complete characterization of the generating chromatic number 
of arbitrary finite solvable groups. 

An interesting direction may be to carry out a finer analysis
of the case of \emph{infinite} monochromatic numbers. To this end,
define $\hat{\chi}_{\mathrm{gen}}(G)$ as the minimal cardinal number
of colors needed to color the elements of $G$ such that there are
no monochromatic generating sets. If this number is finite, then it
is just $\chigen(G)+1$, but in the infinite case, this is the right
definition. Some, but not all, of the proofs in this paper extend
to the infinite case. 

\ed